\documentclass[11pt,a4paper]{article}

\usepackage[utf8]{inputenc}
\usepackage{amsmath, amsthm, amssymb, amsfonts}
\usepackage{geometry}
\usepackage{cite}
\usepackage{hyperref}
\usepackage{mathrsfs}

\newtheorem{theorem}{Theorem}[section]
\newtheorem{lemma}[theorem]{Lemma}
\newtheorem{corollary}[theorem]{Corollary}
\newtheorem{remark}[theorem]{Remark}

\newcommand{\R}{\mathbb{R}}
\newcommand{\Q}{\mathbb{Q}}
\newcommand{\Z}{\mathbb{Z}}
\newcommand{\C}{\mathbb{C}}
\newcommand{\OK}{\mathcal{O}_K}
\newcommand{\OF}{\mathcal{O}_F}

\newcommand{\Vol}{\operatorname{Vol}}

\title{\textbf{On the Natural Density of Monic Integer Polynomials \\ with Roots in a Fixed Number Field}}

\author{
    \textbf{Amirali Fatehizadeh} \\[0.2cm]
    \small \textit{Faculty of Mathematical Sciences, Shahid Beheshti University, Tehran, Iran} \\
    \small \texttt{Amirali.fatehizadeh@gmail.com} \\
    \small \texttt{A.fatehizadeh@mail.sbu.ac.ir}
}
\date{}

\begin{document}

\maketitle

\begin{abstract}
In this article, we investigate the statistical distribution and asymptotic behavior of the family of monic integer polynomials of degree $n$ having at least one root in a fixed number field $K$. Although the framework of thin sets implies that the natural density of this family in the parameter space of bounded height is zero, explicitly quantifying this vanishing rate is a central challenge in arithmetic statistics. Employing a hybrid approach that integrates the Mahler measure, Dirichlet's unit theorem, and residue analysis of the Dedekind zeta function, we demonstrate that the rate of convergence of this density to zero is strictly dependent on the degree $n$. Specifically, we prove that the degrees of the factors induce a phase transition in the asymptotic behavior; for polynomials of degree $n=2$, the decay rate is bounded by $O(H^{-1} \log H)$, whereas for higher degrees, the asymptotic behavior is dominated by the contribution of rational roots, yielding a bound of $O(H^{-1})$. Beyond deriving these asymptotic estimates, we apply principles from the geometry of numbers to establish explicit combinatorial bounds for counting both the reducible and irreducible components of these polynomials. These explicit bounds provide practical tools for computational evaluations within this domain.

\vspace{0.5cm}
\end{abstract}

\section{Introduction}
The statistical study of integer polynomials and the investigation of the distribution of their algebraic properties is an active area of research at the intersection of number theory, algebraic geometry, and Diophantine geometry. A central theme in this branch of arithmetic statistics is determining the density and deriving the precise asymptotic behavior of specific families of polynomials within a parameter space of bounded height. Beyond their fundamental importance in understanding the algebraic structure of number fields, these studies provide precise computational tools and insights for arithmetic dynamics and algorithmic number theory.

The origins of the analytic approach to the distribution of polynomials can be traced back to David Hilbert's seminal work in 1892 and the formulation of Hilbert's Irreducibility Theorem. Hilbert qualitatively demonstrated the existence of infinitely many parametric specializations that preserve irreducibility, a result he subsequently utilized to advance the inverse Galois problem.

The transition from this qualitative perspective to explicit quantification was first achieved by van der Waerden in 1936. He demonstrated that exceptional polynomials (those that are reducible or whose Galois group is not the full symmetric group $S_n$) are extremely rare, and he established the first explicit error bound for their count.

Let $E_n(H)$ denote the number of monic integer polynomials of degree $n$ with coefficients bounded in absolute value by $H$, such that their Galois group is not $S_n$. Hilbert's Irreducibility Theorem implies that $E_n(H) = o(H^n)$. van der Waerden proved that:
\begin{equation*}
    E_n(H) = O\left( H^{n - \frac{1}{6(n-2)\log\log H}} \right),
\end{equation*}
and he further conjectured that $E_n(H) = O(H^{n-1})$. Subsequently, researchers such as Knobloch (1956), Gallagher (1973), Zywina (2010), and Dietmann (2013) improved this bound by employing methods such as the large sieve.

With the introduction and formulation of \textit{thin sets} by Jean-Pierre Serre (1989), this trajectory---some of whose historical milestones were discussed above---underwent significant evolution and achieved structural maturity. Within this framework, the set of polynomials lacking the generic Galois group---including the family investigated in this article---constitutes a thin set; consequently, their density in the space of all polynomials is necessarily zero.

The study of the distribution and counting of polynomial roots over number fields is intrinsically tied to the concept of height. From the perspective of Diophantine geometry, the classical Northcott theorem (1949) serves as the most fundamental principle in this area, guaranteeing the finiteness of algebraic numbers with bounded degree and bounded absolute Weil height.

Although Northcott's theorem is formulated in the space of Weil heights, it possesses a direct connection to the naive height of polynomials via the Mahler measure. When we restrict the search space to polynomials with bounded coefficient height, we are essentially imposing a strict constraint on the Weil height of their roots.

In the problem of counting polynomials that possess a root in a fixed field $K$, this correspondence indicates that the enumeration of roots in $K$ is fundamentally governed by the constraints arising from Northcott's finiteness property. However, since Northcott's theorem is merely an existential result, the central challenge in arithmetic statistics---and the primary objective of this article---is to provide the asymptotic rate and explicit, precise bounds for this behavior in the parameter space.

While the frameworks of Serre's thin sets and Northcott's theorem guarantee a vanishing density, these classical approaches often lack explicit combinatorial upper bounds. Knowing merely that the density is zero is insufficient for algorithmic applications.

In recent years, the study of counting problems concerning integer polynomials with specific algebraic and geometric constraints on their roots has witnessed remarkable progress through a combination of analytic, combinatorial, and arithmetic methods. In this vein, Dubickas and Sha (e.g., in their 2016 and 2018 papers) investigated the density and asymptotic behavior of various families of polynomials in a series of works. For instance, they demonstrated that the set of polynomials having exactly one or two roots of maximal modulus, or those whose roots lie within the neighborhood of a prescribed set of points, possess a positive density. On the other hand, in their work concerning polynomials with multiplicatively dependent roots, they counted the so-called \textit{generalized degenerate polynomials}, obtaining upper and lower bounds, and in certain cases, exact asymptotic formulas. The pivotal tools in these analyses include heights and the Mahler measure; however, the structure of the error terms varies depending on the family under consideration. The success of this framework provides the motivation to revisit classical counting problems in this domain utilizing these modern tools.

Motivated by the analytic precision present in the works of Dubickas and Sha, and leveraging counting principles from number theory, this article provides a direct proof of the vanishing natural density of the targeted polynomials. Specifically, we present an elementary-analytic quantification for the density of monic integer polynomials of degree $n$ having at least one root in a fixed number field.

While bounds for arbitrary number fields are often qualitative, in the special case where $K = \Q$, Chela (1963) and later Kuba (2009) precisely showed that the density of reducible monic polynomials exhibits an asymptotic growth of order $O(H^{n-1})$ (and $O(H\log H)$ for $n=2$). Nevertheless, deriving completely explicit bounds that extend this behavior to arbitrary number fields remains a computational challenge. In this article, not only do we extract these explicit bounds, but we also demonstrate that a behavioral phase transition occurs at degree $n=2$ compared to higher degrees. Furthermore, our analysis establishes the phenomenon of \textit{rational root dominance} in a combinatorial and explicit manner for any arbitrary number field.

The main results of this article are summarized as follows:
\begin{itemize}
    \item \textbf{Deriving asymptotic rates for the reducible case:} We show that the vanishing rate of the density for the target polynomials is bounded by $O(H^{-1})$ for $n > 2$, and by $O(H^{-1}\log H)$ for $n = 2$.
    \item \textbf{Analyzing the irreducible case:} By applying Dirichlet's unit theorem, examining the pole of the Dedekind zeta function, and utilizing the Lipschitz principle on the logarithmic embeddings of the roots, we prove that the contribution of the targeted irreducible polynomials is bounded by $O(H(\log H)^{n-1})$. This establishes an explicit connection between the rank of the field's unit group and the distribution of minimal polynomials in the parameter space.
    \item \textbf{Explicit computational bounds:} We provide an explicit combinatorial upper bound that estimates the number of these polynomials within a bounded box, serving as an evaluable computational tool.
\end{itemize}
\section{Preliminaries}
In this section, we establish the framework of the parameter space, the statistical concepts, and the algebraic tools required to prove the main results.

Let $n \ge 2$ be a fixed integer. We denote the set of all monic integer polynomials of degree $n$ by $\mathcal{P}_n$. We identify each polynomial $P(x) = x^n + a_{n-1}x^{n-1} + \dots + a_0 \in \mathcal{P}_n$ with its coefficient vector $\boldsymbol{a} = (a_0, \dots, a_{n-1}) \in \Z^n$. Accordingly, we define the height of a polynomial as the $\ell^\infty$-norm of its coefficient vector:
\begin{equation*}
    H(P) = \|\boldsymbol{a}\|_\infty = \max_{0 \le i \le n-1} |a_i|.
\end{equation*}
We bound the search space using an integer parameter $H \ge 1$ and define the parametric box $\mathcal{B}_n(H)$ as follows:
\begin{equation*}
    \mathcal{B}_n(H) = \{ P \in \mathcal{P}_n \mid H(P) \le H \}.
\end{equation*}
Clearly, the cardinality of this finite box is $\#\mathcal{B}_n(H) = (2H + 1)^n$. To investigate the statistical distribution of any arbitrary subset $S \subseteq \mathcal{P}_n$, we define the \textit{upper natural density} as follows:
\begin{equation*}
    \overline{d}(S) = \limsup_{H \to \infty} \frac{\#(S \cap \mathcal{B}_n(H))}{\#\mathcal{B}_n(H)}.
\end{equation*}
The upper natural density is finitely subadditive, but it is not countably subadditive in general.

Let $K$ be a fixed number field of degree $D = [K:\Q]$, and let $\OK$ denote its ring of integers. If a polynomial $P \in \mathcal{P}_n$ has a root $\alpha \in K$, then since $P$ is monic with integer coefficients, it must be that $\alpha \in \OK$. Consequently, its minimal polynomial over $\Q$, denoted by $m_\alpha(x)$, is a monic polynomial with integer coefficients. By Gauss's Lemma, since $m_\alpha(x)$ divides $P(x)$ in $\Q[x]$, their quotient must also belong to $\Z[x]$. Therefore, the factorization $P(x) = m_\alpha(x) \cdot Q(x)$ holds over the ring $\Z[x]$, where $Q(x)$ is also a monic polynomial.

To control the asymptotic behavior when bounding the coefficients of these factors, we employ the concept of the Mahler measure. For any polynomial $P(x) = a_n(x - \alpha_1) \dots (x - \alpha_n) \in \C[x]$, its Mahler measure is defined as $M(P) = |a_n| \prod_{i=1}^n \max\{1, |\alpha_i|\}$. One of the well-known inequalities relating the height and the Mahler measure for a polynomial of degree $n$ is:
\begin{equation*}
    H(P) 2^{-n} \le M(P) \le H(P)\sqrt{n+1}.
\end{equation*}
Given the multiplicative property of the Mahler measure and the above inequality, the following classical bound is obtained:

\begin{lemma}[Landau-Mignotte Bound]\label{lem:landau}
Let $P \in \Z[x]$ be a polynomial of degree $n \ge 1$ and let $P(x) = m(x) \cdot Q(x)$ be a factorization, where $m, Q \in \Z[x]$ are polynomials with non-zero leading coefficients. Then there exists a constant $C_n \ge 1$ (depending only on $n$) such that
\begin{equation*}
    H(m) \cdot H(Q) \le C_n H(P).
\end{equation*}
Throughout this paper, we adopt the explicit value $C_n = 2^n \sqrt{n+1}$ as a standard bound.
\end{lemma}

To count the irreducible polynomials with a root in $K$, we need to control the distribution of elements within the ring $\OK$. Each non-zero element $\alpha \in \OK$ generates a principal ideal $\mathfrak{a} = (\alpha)$. The norm of this ideal, defined as $N(\mathfrak{a}) = [\OK : \mathfrak{a}]$, is equal to the absolute value of the algebraic norm of its generator, namely $|N_{K/\Q}(\alpha)|$.

For counting the number of integral ideals of $\OK$ with bounded norm, we use the Dedekind zeta function, defined for complex numbers $s$ with $\operatorname{Re}(s) > 1$ as:
\begin{equation*}
    \zeta_K(s) = \sum_{0 \neq \mathfrak{a} \subseteq \OK} \frac{1}{N(\mathfrak{a})^s},
\end{equation*}
where the sum is over all non-zero integral ideals of $\OK$. The function $\zeta_K(s)$ admits an analytic continuation to the entire complex plane, with a simple pole at $s = 1$. The residue at this pole is given by the class number formula:
\begin{equation*}
    \kappa_K := \lim_{s \to 1} (s - 1)\zeta_K(s) = \frac{2^{r_1}(2\pi)^{r_2}h_K R_K}{w_K \sqrt{|D_K|}},
\end{equation*}
where $r_1$ and $r_2$ are the number of real and complex embeddings of $K$, respectively; $h_K$ is the class number, $R_K$ is the regulator, $w_K$ is the number of roots of unity in $K$, and $D_K$ is the discriminant of $K$.

Applying a standard Tauberian theorem (due to Landau or Weber) to this simple pole shows that if $I_K(X)$ denotes the number of integral ideals of $\OK$ with norm at most $X$, then the following asymptotic behavior holds:
\begin{equation*}
    I_K(X) = \kappa_K X + O\left(X^{1 - 1/D}\right).
\end{equation*}
In our counting arguments, since the algebraic norm of an algebraic integer $\alpha$ is proportional to the height of its minimal polynomial, this asymptotic result ensures that the number of permissible principal ideals generated by such integers is at most of order $O(H)$.

Fixing an ideal $(\alpha)$ does not uniquely determine its generator; rather, any other generator of this ideal is of the form $\alpha' = \alpha u$ for some unit $u \in \OK^\times$. To count the elements precisely, we must enumerate the number of units that remain permissible under the height constraints of the problem.

According to Dirichlet's Unit Theorem, this group has the isomorphic structure $\OK^\times \cong W_K \times \Z^{r_1+r_2-1}$, where $W_K$ is the finite cyclic group of roots of unity contained in $K$. It is well-known that the rank of the unit group is $r_1 + r_2 - 1$. 

Furthermore, we utilize the logarithmic embedding:
\begin{equation*}
    \log: \OK^\times \longrightarrow \R^{r_1+r_2}, \qquad u \mapsto \left( \log|\sigma_1(u)|, \dots, \log|\sigma_{r_1}(u)|, 2\log|\sigma_{r_1+1}(u)|, \dots, 2\log|\sigma_{r_1+r_2}(u)| \right)
\end{equation*}
The image of $\OK^\times$ under this map is a lattice of rank $r_1 + r_2 - 1$ within a hyperplane in the Euclidean space $\R^{r_1+r_2}$.

When an irreducible minimal polynomial $P$ with bounded height $H(P) \le H$ generates a root $\alpha \in \OK$, this root defines a subfield $F = \Q(\alpha)$ of degree $n = [F:\Q]$. Since $P$ is monic, its constant term is equal to $\pm N_{F/\Q}(\alpha)$, which implies $|N_{F/\Q}(\alpha)| \le H$. Note that if the constant term were zero, then $x$ would be a factor of $P(x)$, contradicting its irreducibility.

Furthermore, by the classical Cauchy bound for the roots of a polynomial, for any embedding $\sigma_i : F \hookrightarrow \C$, we have:
\begin{equation*}
    |\sigma_i(\alpha)| \le 1 + H \le 2H.
\end{equation*}
Using the multiplicative property of the norm, for any fixed embedding $\sigma_i$, we can write $|\sigma_i(\alpha)| = \frac{|N_{F/\Q}(\alpha)|}{\prod_{j \neq i} |\sigma_j(\alpha)|}$. Since $|N_{F/\Q}(\alpha)| \ge 1$ (as $\alpha \in \OK$ is non-zero), and by the upper bound on the other conjugates, we deduce a lower bound $\prod_{j \neq i} |\sigma_j(\alpha)| \ge (2H)^{n-1}$. Combining the upper and lower bounds yields:
\begin{equation*}
    \frac{1}{(2H)^{n-1}} \le |\sigma_i(\alpha)| \le 2H.
\end{equation*}
Taking logarithms on all sides, we obtain for each embedding $\sigma_i$:
\begin{equation*}
    -(n-1)\log(2H) \le \log|\sigma_i(\alpha)| \le \log(2H).
\end{equation*}
Thus, the logarithmic vector $(\log|\sigma_1(\alpha)|, \dots, \log|\sigma_n(\alpha)|)$ is confined within a hypercube in $\R^n$ where the side lengths are of magnitude $O(\log H)$.

Based on the geometry of numbers, and by applying the Lipschitz principle for counting lattice points in parametrically bounded domains, the number of lattice points (corresponding to permissible units) inside this region is proportional to its volume. Since the dimension of the unit lattice is $r_1 + r_2 - 1$, the number of permissible units is bounded by $O((\log H)^{r_1+r_2-1})$. Consequently, this component contributes a growth rate of order $O((\log H)^{n-1})$ to our calculations, as $n = r_1 + 2r_2 = D$ for $F \subseteq K$. 

This decomposition into first counting ideals and then counting units within the logarithmic lattice constitutes the core of our analysis for counting the targeted irreducible polynomials.

Throughout this article, for two non-negative functions $U$ and $V$, the notation $U \ll V$ or $U = O(V)$ means that the inequality $U \le cV$ holds for some positive constant $c$. This constant is strictly independent of the height parameter $H$, though it may depend on fixed system parameters such as the polynomial degree, the field degree, and other field invariants.\section{Proof of the Main Theorem}
In this section, we provide a straightforward proof of the main theorem of the article. Our strategy is based on an asymptotic analysis, direct counting, and explicit bounding of the target polynomials within the parametric box $\mathcal{B}_n(H)$.

\begin{theorem}\label{thm:main}
Let $K$ be a fixed number field. The natural density of the set of monic integer polynomials of degree $n$ that have at least one root in $K$ is zero.
\end{theorem}

\begin{proof}
Let $M_K(H) \subseteq \mathcal{B}_n(H)$ denote the set of all target polynomials. As discussed in the previous section, the existence of a root $\alpha \in K$ for a polynomial $P \in M_K(H)$ implies that $\alpha \in \OK$, and its minimal polynomial, $m_\alpha(x)$, must be a factor of $P(x)$ in the ring $\Z[x]$.

We partition the set $M_K(H)$ into two disjoint subsets based on its reducibility behavior over the field of rational numbers:
\begin{equation*}
    M_K(H) = \mathrm{red}_n(H) \cup \mathrm{irr}_n(H),
\end{equation*}
where $\mathrm{red}_n(H)$ comprises the reducible polynomials and $\mathrm{irr}_n(H)$ comprises the irreducible ones. We will bound the cardinality of each of these sets independently.

First, consider a polynomial $P \in \mathrm{red}_n(H)$. Such a polynomial admits a factorization $P(x) = m(x) \cdot Q(x)$ over $\Z[x]$, where $m(x)$ is the minimal polynomial of a root $\alpha \in \OK$ with degree $\deg(m) = d$. Since $P$ is reducible, the degree of this factor must lie in the range $1 \le d \le \min(n-1, D)$, where $D =[K:\Q]$. 

By Lemma \ref{lem:landau}, the height of the quotient $Q(x)$, which is a monic polynomial of degree $n-d$, satisfies the following inequality:
\begin{equation*}
    H(Q) \le \frac{C_n H(P)}{H(m)} \le \frac{C_n H}{H(m)}.
\end{equation*}
To precisely count the possible pairs $(m, Q)$, assume that the height of the minimal polynomial is fixed at $H(m) = k$.
\begin{itemize}
    \item \textbf{Number of choices for $m(x)$:} The monic polynomial $m(x)$ is determined by its $d$ non-leading coefficients. For its height to be exactly $k$, all its coefficients must lie in the interval $[-k, k]$, with at least one coefficient being exactly $\pm k$. The total number of such polynomials is exactly $(2k+1)^d - (2k-1)^d$, which exhibits an asymptotic behavior of $O(k^{d-1})$.
    \item \textbf{Number of choices for $Q(x)$:} Similarly, the monic polynomial $Q(x)$ is determined by its $n-d$ non-leading coefficients, each of which is bounded within the interval $[-C_n H/k, C_n H/k]$. Thus, the number of possible choices for $Q(x)$ is of the order $O\left( \left(\frac{H}{k}\right)^{n-d} \right)$.
\end{itemize}

By summing the product of these bounds over all possible heights $1 \le k \le C_n H$ and all permissible degrees $1 \le d \le \min(n-1, D)$, we obtain the following upper bound for the cardinality of the reducible component:
\begin{equation*}
    \#\mathrm{red}_n(H) \ll \sum_{d=1}^{\min(n-1, D)} \sum_{k=1}^{C_n H} k^{d-1} \cdot \left(\frac{H}{k}\right)^{n-d} = \sum_{d=1}^{\min(n-1, D)} H^{n-d} \sum_{k=1}^{C_n H} k^{2d-n-1}.
\end{equation*}

To evaluate the asymptotic behavior of this expression, we consider the exponent of the variable $k$ in the inner sum, denoted by $p = 2d - n - 1$. Note that $p$ is an integer. The asymptotic behavior of the series $\sum k^p$ depends entirely on the sign and value of $p$. We analyze this series in three distinct cases:

\begin{itemize}
    \item \textbf{Case 1 ($p < -1$):} In this regime, the series $\sum k^p$ converges. Consequently, the sum remains bounded by a constant independent of $H$, i.e., $O(1)$, regardless of how large $H$ becomes. Thus, the total contribution of this part is $H^{n-d} \cdot O(1) \ll H^{n-1}$.
    
    \item \textbf{Case 2 ($p = -1$):} This case occurs only if $n$ is even and $d = n/2$. Here, the sum reduces to a harmonic series, which exhibits logarithmic divergence, namely $\sum_{k=1}^{C_n H} k^{-1} = O(\log H)$. The contribution of this section is $H^{n-d}\log H$. For the specific case $n=2$ (where $d=1$), this yields a contribution of $O(H \log H)$. For $n \ge 4$, since $d = n/2 \le n-2$, this contribution is also dominated by $O(H^{n-1})$.
    
    \item \textbf{Case 3 ($p > -1$):} Under this condition, the series diverges, and its asymptotic behavior, based on an integral approximation, is $O(H^{p+1}) = O(H^{2d-n})$. Therefore, the contribution of this part is $H^{n-d} \cdot H^{2d-n} = H^d$. Since the maximum possible value for $d$ is $n-1$, this case is also strictly bounded by $O(H^{n-1})$.
\end{itemize}

Combining these three cases, we conclude that:
\begin{equation*}
    \#\mathrm{red}_n(H) \ll 
    \begin{cases} 
      H \log H & \text{if } n = 2, \\
      H^{n-1} & \text{if } n > 2.
    \end{cases}
\end{equation*}

Now, consider the case where $P \in \mathrm{irr}_n(H)$. In this scenario, the polynomial $P$ is irreducible over $\Q$ and possesses a root $\alpha \in \OK$. Consequently, $P$ must be the minimal polynomial of $\alpha$, and its degree is exactly $n$. The root $\alpha$ generates a subfield $F = \Q(\alpha) \subseteq K$ with degree $[F:\Q] = n$. Since the base field $K$ is fixed, the number of such subfields is finite; we denote the set of these subfields by $\mathrm{Sub}_n(K)$. Note that if $[K:\Q] < n$, the set of irreducible polynomials satisfying this condition is empty. Furthermore, each minimal polynomial of degree $n$ produces exactly $n$ conjugate roots. Therefore, counting the number of permissible elements $\alpha$ provides a valid upper bound for the number of these polynomials.

We first perform the count for a fixed subfield $F \in \mathrm{Sub}_n(K)$. The element $\alpha$ generates a principal ideal $\mathfrak{a} = (\alpha)$ in the ring of integers $\OF$. The algebraic norm of this root is equal to the absolute value of the constant term of the polynomial, i.e., $|N_{F/\Q}(\alpha)| = |P(0)| \le H$. As established in the previous section, the total number of possible principal ideals satisfying this norm constraint is at most of order $O(H)$.

Any other generator for these ideals is of the form $\alpha' = \alpha u$, where $u \in \OF^\times$ is a unit. By Dirichlet's Unit Theorem, since $[F:\Q] = n$, the rank of the unit group is $r = r_1 + r_2 - 1 \le n - 1$. By applying the logarithmic embedding, and given that all coefficients of $P$ are bounded by $H$, the logarithmic vector of the roots is confined within a bounded region with side lengths of order $O(\log H)$. As discussed earlier, the number of permissible units in this lattice region is at most $O((\log H)^{n-1})$.

By multiplying the number of ideals by the number of permissible units, the total number of target irreducible polynomials within the subfield $F$ is bounded by $O(H(\log H)^{n-1})$. Finally, by summing this bound over the finite set of subfields $\mathrm{Sub}_n(K)$, and noting that the cardinality of this set is $O(1)$ relative to $H$, the total number of permissible irreducible polynomials is bounded as follows:
\begin{equation*}
    \#\mathrm{irr}_n(H) \ll \sum_{F \in \mathrm{Sub}_n(K)} H(\log H)^{n-1} \ll H(\log H)^{n-1}.
\end{equation*}

Having determined the cardinalities of both components, the upper natural density of the entire target family is given by:
\begin{equation*}
    \overline{d}(M_K) = \limsup_{H \to \infty} \frac{\#\mathrm{red}_n(H) + \#\mathrm{irr}_n(H)}{\#\mathcal{B}_n(H)} \ll \limsup_{H \to \infty} \frac{O(H^{n-1}) + O(H(\log H)^{n-1})}{O(H^n)}.
\end{equation*}
Consequently, the above limit tends to zero. Since the natural density is a non-negative quantity, it follows that the natural density exists and is equal to zero.
\end{proof}

\begin{remark}
The asymptotic analysis of the summation over the degrees of the factor $m(x)$ reveals an intriguing structural property. Since the leading power of $H$ is generated by $\max(d, n-d)$, the absolute maximum exponent always occurs at $d=1$, which yields the $O(H^{n-1})$ behavior. Algebraically, this case corresponds to the situation where the root $\alpha$ is a rational integer (belonging to the base subfield $\Q$). This result indicates that within the parameter space, the primary drivers for generating polynomials with roots in $K$ are, in fact, polynomials with rational roots, whereas more complex roots ($d > 1$) contribute significantly less. Furthermore, it should be noted that if the degree of the base field is smaller than the degree of the polynomial ($D < n-1$), our bound is markedly improved and becomes of the order $O(H^{n-D})$.
\end{remark}\section{Effective and Explicit Asymptotic Bounds}
While Theorem \ref{thm:main} characterizes the statistical distribution of the targeted polynomials from a theoretical perspective, in computational applications, access to explicit upper bounds and knowledge of the precise rate of convergence to zero is of particular importance. In this section, we determine the exact convergence rates and provide independent effective bounds for the reducible and irreducible components of the search space.

\begin{corollary}[Asymptotic Convergence Rate]
By comparing the bounds derived in Section 3, it is observed that for $n \ge 3$, the contribution of irreducible polynomials is negligible compared to that of the reducible ones. By dividing the dominant term by the total cardinality of the search space, the vanishing rate of the natural density for degrees $n = 2$ and $n \ge 3$ is $O(H^{-1}\log H)$ and $O(H^{-1})$, respectively.
\end{corollary}

\begin{corollary}[Explicit Computational Bound]
Let $n \ge 2$ be an integer and $K$ be a fixed number field of degree $D = [K:\Q]$. Let $\#\mathrm{red}_n(H)$ denote the exact number of reducible monic integer polynomials of degree $n$ in the box $\mathcal{B}_n(H)$ having at least one root in $K$. Then, for any $H \ge 1$, the following explicit bound holds:
\begin{equation*}
    \#\mathrm{red}_n(H) \le \sum_{d=1}^{\min(n-1, D)} \sum_{k=1}^{C_n H} \Big[ (2k+1)^d - (2k-1)^d \Big] \left( 2 \left\lfloor \frac{C_n H}{k} \right\rfloor + 1 \right)^{n-d},
\end{equation*}
where $C_n = 2^n \sqrt{n+1}$ is the Landau-Mignotte constant.
\end{corollary}
The proof of this result follows directly from the combinatorial arguments regarding the height constraints of the factors presented in Section 3.

\begin{corollary}
Suppose $\mathrm{Sub}_n(K) \neq \emptyset$. Then the cardinality of the irreducible monic integer polynomials is bounded by the following expression
\begin{equation*}
    \#\mathrm{irr}_n(H) \le \sum_{F \in \mathrm{Sub}_n(K)} \left( \kappa_F H + C_{1,F} H^{1 - 1/n} \right) \times \left( \frac{w_F}{R_F} \cdot \Vol(S_H(F)) + C_{2,F} \right)
\end{equation*}
where for each subfield $F \in \mathrm{Sub}_n(K)$:
\begin{itemize}
    \item $\kappa_F$ is the residue of the Dedekind zeta function of $F$ at $s = 1$.
    \item $\Vol(S_H(F))$ is the volume of the convex region in the logarithmic unit lattice of $\OF^\times$ bounded by the Cauchy root bound proportional to $H$.
    \item $C_{1,F}$ and $C_{2,F}$ are absolute constants representing the counting error terms associated with the ideal counting theorem and the Lipschitz principle, respectively. Both are effectively computable (e.g., $C_{1,F}$ can be explicitly derived using effective versions of Landau's theorem such as the work of Sunley (1973), and $C_{2,F}$ can be obtained from explicit versions of the Lipschitz principle in the geometry of numbers, such as the work of Widmer (2012)).
\end{itemize}
\end{corollary}

\begin{proof}
Every polynomial in $\mathrm{irr}_n(H)$ is a minimal polynomial for some element $\alpha \in \OF$, where $F \in \mathrm{Sub}_n(K)$ is a subfield of degree $n$. Counting these polynomials is equivalent to counting the permissible elements $\alpha$ within each subfield $F$. The element $\alpha$ generates a principal ideal $\mathfrak{a} = (\alpha)$ in $\OF$. Since the algebraic norm of this ideal is bounded by $H$, the explicit version of the Weber (or Landau) ideal counting theorem ensures that the number of such ideals in $\OF$ is controlled by $\kappa_F H + C_{1,F} H^{1 - 1/n}$. For each fixed ideal, the number of units $u \in \OF^\times$ that keep the element $\alpha u$ within the height constraint $H$ corresponds to counting lattice points in the logarithmic lattice of $F$ inside a bounded region of volume $\Vol(S_H(F))$. Based on Dirichlet's unit theorem and volume approximation, the number of such points is bounded by $\frac{w_F}{R_F} \cdot \Vol(S_H(F))$ plus a boundary error term $C_{2,F}$. Multiplying these two independent bounds for each subfield and summing over the set $\mathrm{Sub}_n(K)$ yields the explicit bound.
\end{proof}


\end{document}